\documentclass[11pt]{amsart}
\pdfoutput=1

\usepackage{amssymb}
\usepackage{microtype}
\usepackage{tikz}
\usetikzlibrary{arrows}
\usepackage{booktabs}
\usepackage[pdftex,colorlinks,citecolor=black,linkcolor=black,urlcolor=black,bookmarks=false]{hyperref}

\usepackage{doi}

\newtheorem{theorem}{Theorem}[section]
\newtheorem{proposition}[theorem]{Proposition}

\theoremstyle{definition}

\numberwithin{figure}{section}
\numberwithin{equation}{section}
\numberwithin{table}{section}

\newcommand{\R}{\mathbb{R}}

\newcommand{\SL}{\operatorname{SL}}
\newcommand{\GL}{\operatorname{GL}}

\newcommand{\SO}{\operatorname{SO}}
\newcommand{\Sp}{\operatorname{Sp}}

\newcommand{\tr}{\operatorname{tr}}

\title{Fricke's trace identity and spin groups}

\author{Matthew de Courcy-Ireland}
\address{Department of Mathematics\\
Stockholm University} \email{matthew.decourcy-ireland@math.su.se}

\date{July 20, 2026}

\begin{document}

\maketitle
\begin{abstract}
We give a proof of Fricke's trace identity using the exceptional spin double cover of an orthogonal group in four variables.
We also explain how Fricke's identity is related to the double-angle formula from trigonometry as well as an identity for symplectic matrices.
\end{abstract}
\section{Introduction}

One of the low-dimensional exceptions of Lie theory is the spin double cover from $\SL_2 \times \SL_2$ to an orthogonal group in four variables.
The purpose of this note is to use this feature of $\SL_2 \times \SL_2$ to give a proof of Fricke's trace identity.
This identity holds, over any commutative ring with $1$, for any pair of matrices $(A,B)$ in $\SL_2 \times \SL_2$.
It states that
\begin{equation} \label{eqn:fricke}
(\tr{A})^2+(\tr{B})^2+(\tr{AB})^2 = (\tr{A})(\tr{B})(\tr{AB})+\tr(ABA^{-1}B^{-1})+2.
\end{equation}
Fricke's identity can be reformulated as
\begin{equation} \label{eqn:fricke-symm}
(\tr{A})^2+(\tr{B})^2-2 = \tr(AB)\big((\tr{A})(\tr{B})-\tr(AB)\big)+\tr(ABA^{-1}B^{-1})
\end{equation}
which helps to underscore some of the cancellation involved.
Each term on the left is invariant under separate conjugations $(A,B) \mapsto (g_1^{-1}Ag_1, g_2^{-1}Bg_2)$
whereas the terms on the right are not individually preserved unless $g_1=g_2$.
Nevertheless, combining these terms recovers the full symmetry.

Fricke's identity plays a fundamental role in the theory of character varieties and Markoff cubic surfaces (see for example \cite{goldman}).
A noncommutative generalisation has been given in \cite[Corollary 4.7]{GKW}.
For the standard textbook proof of Fricke's identity, one can refer to \cite[p.65, proof of Proposition 4.3]{aigner}.
It obtains (\ref{eqn:fricke}) by repeatedly applying the relation (for $X$, $Y$ in $\SL_2$)
\begin{equation} \label{eqn:trace-inverse}
\tr(XY)+\tr(XY^{-1})=(\tr{X})(\tr{Y}).
\end{equation}

The proof given here instead deduces (\ref{eqn:fricke-symm}) by comparing two different ways of calculating the characteristic polynomial of the spin image of $(A,B)$.
In Section~\ref{sec:sl2sl2spin}, we review the interpretation of $\SL_2 \times \SL_2$ as a spin group.
In Section~\ref{sec:profricke}, we calculate the characteristic polynomial and prove Fricke's identity.
In Section~\ref{sec:smaller-examples}, we show how this approach is related to the classical double-angle formula for cosines, and to the lower-dimensional exceptional spin covers.
Finally, in Section~\ref{sec:symplectic}, we discuss how Fricke's identity fits into the last of the exceptional spin covers.

\section{$\SL_2 \times \SL_2$ as a spin group} \label{sec:sl2sl2spin}

In this section, we recall the exceptional spin covering $\SL_2 \times \SL_2 \rightarrow SO(2,2)$.
The orthogonal group $\SO_n(\R)$ has a double cover, called a spin group.
For $n \leq 6$, there are exceptional isomorphisms between these spin groups and other Lie groups.
We are interested especially in $n=4$ and, although we are not necessarily working over the real numbers, we will see that signature $(2,2)$ comes naturally from the construction (see (\ref{eqn:det-22}) below).
We will meet the cases $n=5,6$ again later in Section~\ref{sec:symplectic}.

In any dimension, $\SL_n \times \SL_n$ preserves the determinant:
\[
\det(g_1 m g_2) = \det(g_1) \det(m) \det(g_2) = \det(m)
\]
for any $n\times n$ matrix $m$.
The special feature of $n=2$ compared to higher dimensions is that $\det(m)$ is a quadratic form. The action therefore gives a map from $\SL_2 \times \SL_2$ to an orthogonal group.

To keep the action on the left, we define
\[
(g_1,g_2) \cdot m = g_1 m g_2^{-1}
\]
where $\det(m)$ is preserved by $\SL_n$.
Scalars act trivially.
For $\SL_2$, this means $(-g_1,-g_2)$ acts the same way as $(g_1,g_2)$.

We will take $(p,q,r,s)$ in that order as the variables in the quadratic form
\[
ps-qr = \det \begin{pmatrix} p & q \\ r & s \end{pmatrix}.
\]
The action is given by
\begin{align*}
g_1 m g_2^{-1} &= \begin{pmatrix} a & b \\ c & d \end{pmatrix} \begin{pmatrix} p & q \\ r & s \end{pmatrix} \begin{pmatrix} d' & -b' \\ -c' & a' \end{pmatrix} \\
&= \begin{pmatrix} apd' + brd' - aqc' -bsc' & -apb' -brb'+aqa'+bsa' \\ cpd' + drd' -cqc' -dsc' & -cpb'-drb'+cqa'+dsa' \end{pmatrix}
.
\end{align*}
On column vectors, $(g_1,g_2)$ acts by
\[
\begin{pmatrix} p \\ q \\ r \\ s \end{pmatrix}
\mapsto
\begin{pmatrix} ad' & - ac' & bd' & -bc' \\ -ab' & aa' & -bb' & ba' \\ cd' & -cc' & dd' & -dc' \\ -cb' & ca' & -db' & da' \end{pmatrix}
\begin{pmatrix} p \\ q \\ r \\ s \end{pmatrix}
.
\]
We will write
\begin{equation} \label{eqn:u}
u=u(g_1,g_2) = \begin{pmatrix} ad' & - ac' & bd' & -bc' \\ -ab' & aa' & -bb' & ba' \\ cd' & -cc' & dd' & -dc' \\ -cb' & ca' & -db' & da' \end{pmatrix} = \big[u_{ij} \big]_{i,j=1}^{4}
\end{equation}
or in block form
\begin{equation} \label{eqn:u-block}
u = \begin{pmatrix} a g_2^{-\top} & b g_2^{-\top} \\ c g_2^{-\top} & dg_2^{-\top}  \end{pmatrix}
= g_1 \otimes g_2^{-\top}.
\end{equation}
It is possible to recover $g_1$ and $g_2$ up to a single choice of sign.
In the top row of (\ref{eqn:u}), from $u_{11}=ad'$ and $u_{13}=bd'$ we deduce that
\[
a = \frac{u_{11}}{d'} = \frac{u_{11}}{u_{13}}b.
\]
We can equally well solve for $d'$ in terms of $c$ or $d$:
\[
a= \frac{u_{11}}{u_{31}} c = \frac{u_{11}}{u_{33}} d.
\]
We substitute these into $ad-bc=1$ to solve for $a$:
\begin{align*}
1 = ad-bc = a^2 \Big( \frac{u_{33}}{u_{11}}-\frac{u_{13}u_{31}}{u_{11}^2} \Big) = a^2 u_{11}^{-2} (u_{11}u_{33}-u_{13}u_{31})
.
\end{align*}
There are two solutions for
\begin{equation} \label{eqn:2to1a}
a= \pm \frac{u_{11}}{\sqrt{u_{11}u_{33}-u_{13}u_{31}}},
\end{equation}
with all other entries then determined:
\begin{equation} \label{eqn:u-to-g}
\begin{pmatrix} a & b \\ c & d \end{pmatrix} = \frac{a}{u_{11}} \begin{pmatrix} u_{11} & u_{13} \\ u_{31}& u_{33} \end{pmatrix}, 
\quad
\begin{pmatrix} a' & b' \\ c' & d' \end{pmatrix} = a^{-1} \begin{pmatrix} u_{22} & -u_{21} \\ -u_{12} & u_{11} \end{pmatrix}
.
\end{equation}

There is a whole array of other ways we could have expressed the remaining variables as multiples of $a$.
\[
\begin{aligned}
b&=\frac{u_{13}}{d'}=\frac{u_{13}}{u_{11}} a,& c&=\frac{u_{31}}{d'}=\frac{u_{31}}{u_{11}} a,& d&=\frac{u_{33}}{d'}=\frac{u_{33}}{u_{11}} a& \\
&=-\frac{u_{14}}{c'}=\frac{u_{14}}{u_{12}} a,& &=-\frac{u_{32}}{c'}=\frac{u_{32}}{u_{12}} a,& &=-\frac{u_{34}}{c'}=\frac{u_{34}}{u_{12}} a& \\
&=-\frac{u_{23}}{b'}=\frac{u_{23}}{u_{21}} a,& &=-\frac{u_{41}}{b'}=\frac{u_{41}}{u_{21}} a,& &=-\frac{u_{43}}{b'}=\frac{u_{43}}{u_{21}} a& \\
&=\frac{u_{24}}{a'}=\frac{u_{24}}{u_{22}} a,& &=\frac{u_{42}}{a'}=\frac{u_{42}}{u_{22}} a,& &=\frac{u_{44}}{a'}=\frac{u_{44}}{u_{22}} a&
\end{aligned}
\]
This implies 9 relations among the entries $u_{ij}$
\begin{equation} \label{eqn:9-quadratic-relations}
\left\{
\begin{aligned}
&u_{13}u_{12}=u_{11}u_{14},& &u_{13}u_{21}=u_{11}u_{23},& &u_{13}u_{22}=u_{11}u_{24}& \\
&u_{31}u_{12}=u_{11}u_{32},& &u_{31}u_{21}=u_{11}u_{32},& &u_{31}u_{22}=u_{11}u_{42}& \\
&u_{33}u_{12}=u_{11}u_{34},& &u_{33}u_{21}=u_{11}u_{43},& &u_{33}u_{22}=u_{11}u_{44}&
\end{aligned}
\right.
\end{equation}
A tenth relation follows from $ad-bc=1$, namely
\begin{equation} \label{eqn:det-tenth}
(u_{22}u_{11}-u_{21}u_{12})(u_{11}u_{33}-u_{13}u_{31})=u_{11}^2
.
\end{equation}
Just as $\SL_2 \times \SL_2$ is 6-dimensional, these 10 equations in the 16 variables $u_{ij}$ define a 6-dimensional group.
One can of course eliminate $u_{11}$ in various ways, for instance equating $u_{11}=u_{13}u_{12}/u_{14}$ with $u_{11}=u_{13}u_{21}/u_{23}$ to get
\[
u_{12}u_{23}=u_{13}u_{14}.
\]

\section{Proof of Fricke's identity} \label{sec:profricke}

In this section, we use the map (\ref{eqn:u}) from $\SL_2 \times \SL_2$ to $SO(\det)$ from Section~\ref{sec:sl2sl2spin} to prove (\ref{eqn:fricke-symm}).
The idea is to calculate the characteristic polynomial of $u$ in two ways.
We will show in Proposition~\ref{prop:charpoly} that the coefficient of $x^2$ is
\[
(\tr{A})^2+(\tr{B})^2-2.
\]
On the other hand, for any $n \times n$ matrix $u$, one can write the coefficient of $x^{n-2}$ in $\det(xI-u)$ as
\[
\sum_{i<j} (u_{ii}u_{jj}-u_{ij}u_{ji}).
\]
We will show that this equals the other side of Fricke's identity.

\subsection{Palindromes}

Fricke's identity is closely related to the fact that $\tr(g)=\tr(g^{-1})$ for matrices in $\SL_2$, which underlies (\ref{eqn:trace-inverse}) and follows from it with $X=I$, $Y=g$.
To see something similar for larger matrices, let us recall that characteristic polynomials enjoy the following symmetry.

\begin{proposition} \label{prop:palin}
For any invertible matrix $g$, the characteristic polynomials of $g$ and $g^{-1}$ are related by
\[
\det(xI-g^{-1}) = \frac{(-1)^n}{\det{g}} x^n \det(x^{-1}I-g).
\]
\end{proposition}
\begin{proof}
At the level of matrices,
\[
x-g^{-1} = -g^{-1} x (x^{-1}-g)
\]
and the result follows by taking determinants of both sides.
\end{proof}

In the special case of an orthogonal matrix, it follows that the characteristic polynomial is a palindrome up to sign.
\begin{proposition} \label{prop:palindrome}
If $u^{\top}Qu = Q$ for some invertible matrix $Q$, then
\[
\det(xI-u) = \frac{(-x)^n}{\det{u}} \det(x^{-1}I-u)
.
\]
\end{proposition}
\begin{proof}
Apply the previous proposition with $g=u$, using the special feature that $u^{-1}=Q^{-1}u^{\top}Q$ is conjugate to $u$.

More simply, without even using the fact that $u$ is similar to $u^{\top}$, it is enough to know that a matrix and its transpose have the same determinant. Therefore
\[
\det(x-u^{-1}) = \det(Q^{-1}(x-u^{\top})Q) = \det(x-u^{\top}) = \det(x-u).
\]
\end{proof}

In the orthogonal case, $\det{u}=\pm 1$ by taking determinants of both sides of $u^{\top}Qu=Q$, assuming $Q$ is non-degenerate and working over a field.
Therefore the polynomial $p(x)=\det(xI-u)$ satisfies $p(x)=\pm x^n p(x^{-1})$.
In the cases of interest, $n$ is even and $u \in SO(Q)$. Then $p$ is a palindrome.

\subsection{First way of calculating the characteristic polynomial}

\begin{proposition} \label{prop:charpoly}
For any $(A,B) \in \SL_2 \times \SL_2$, the image $u=u(A,B)$ from (\ref{eqn:u}) has characteristic polynomial $\det(xI-u)$ equal to
\[
x^4 - (\tr{A})(\tr{B})x^3 + \big( (\tr{A})^2+(\tr{B})^2-2) \big)x^2 - (\tr{A})(\tr{B})x + 1
.
\]
\end{proposition}

The proof also gives a result for matrices in $\GL_2$ instead of $\SL_2$.
Starting from matrices with determinants $\Delta$, $\Delta'$ and traces $\tau$, $\tau'$, the characteristic polynomial would be
\[
x^4 - \tau \tau' x^3 + (\Delta' \tau^2 + \Delta \tau'^2 - 2\Delta \Delta') x^2 - \Delta \Delta' \tau \tau' x + \Delta^2 \Delta'^2.
\]

\begin{proof}
We expand the determinant $\det(xI-u)$ and see what simplifies in view of
\[
ad-bc=a'd'-b'c' = 1.
\]
For the coefficient of $x^3$, we have
\[
\tr(u) = ad'+aa'+dd'+da' = (a+d)(a'+d')=(\tr{A})(\tr{B})
\]
as required.
The constant term is
\begin{align*}
(a'^2d'^2-2a'b'c'd'+b'^2c'^2)a^2d^2 &+ (a'^2d'^2-2a'b'c'd'+b'^2c'^2)b^2c^2 \\
&+ (-2a'^2d'^2+4a'b'c'd'-2b'^2c'^2)abcd
\end{align*}
which can be factored as
\[
(a'd'-b'c')^2(a^2d^2+b^2c^2-2abcd) = (a'd'-b'c')^2(ad-bc)^2=1.
\]

The coefficient of $x$, for $u$ in the orthogonal group, must be the same as the coefficient of $x^3$ by Proposition~\ref{prop:palindrome}.
We can also check directly that the coefficient of $x$ is
\begin{align*}
(-a'd'^2+b'c'd'-a'^2d'+a'b'c' ) ad^2 + \\
(a'd'^2-b'c'd'+a'^2d'-a'b'c' ) bcd + \\
(-a'd'^2+b'c'd'-a'^2d'+a'b'c' ) a^2d \\
(a'd'^2-b'c'd'+a'^2d'-a'b'c' ) abc
\end{align*}
hence
\[
(-a'd'^2+b'c'd'-a'^2d'+a'b'c')(ad^2-bcd+a^2d-abc).
\]
One factor is
\begin{align*}
ad^2-bcd+a^2d-abc &= (ad-bc)d+a(ad-bc) \\
&= (ad-bc)(a+d)
\end{align*}
and similarly for $a',b',c',d'$.
The coefficient of $x$ is therefore
\[
(ad-bc)(a'd'-b'c')(a+d)(a'+d') = (\tr{A})(\tr{B}).
\]
This is the same as for $x^3$, using $ad-bc=a'd'-b'c'=1$.

Finally, the coefficient of $x^2$ is
\begin{align*}
&(a'd'-b'c')d^2 + (d'^2+2a'd'+a'^2)ad -(d'^2+2b'c'+a'^2)bc + (a'd'-b'c')a^2 \\
&=(a'd'-b'c')(a^2+d^2) + (a'+d')^2ad - \big((a'+d')^2-2a'd'\big)bc-2b'c'bc \\
&=(a'd'-b'c')(a^2+d^2) + (a'+d')^2 (ad-bc) + 2(a'd'-b'c')bc \\
&=(a'd'-b'c')\big((a+d)^2-2ad\big) + (a'+d')^2 (ad-bc) + 2(a'd'-b'c')bc \\
&=(a'd'-b'c')(a+d)^2+(ad-bc)(a'+d')^2 - 2(a'd'-b'c')(ad-bc)
\end{align*}
Since $ad-bc=a'd'-b'c'=1$, this is
\[
(a+d)^2+(a'+d')^2-2 = (\tr{A})^2+(\tr{B})^2-2.
\]
\end{proof}

\subsection{Second way of calculating the characteristic polynomial}

Our task is to show
\[
\sum_{i<j} (u_{ii}u_{jj}-u_{ij}u_{ji})=\big((\tr{A})(\tr{B}) - \tr(AB) \big) \tr(AB) + \tr(ABA^{-1}B^{-1})
\]
where
\[
A = \begin{pmatrix} a & b \\ c & d \end{pmatrix}, \quad B=\begin{pmatrix} a' & b' \\ c' & d' \end{pmatrix}, 
\quad
u=\begin{pmatrix} a \begin{pmatrix} d' & -c' \\ -b' & a' \end{pmatrix} & b \begin{pmatrix} d' & -c' \\ -b' & a' \end{pmatrix} \\ c \begin{pmatrix} d' & -c' \\ -b' & a' \end{pmatrix} & d \begin{pmatrix} d' & -c' \\ -b' & a' \end{pmatrix} \end{pmatrix}
\]
in the same notation as before.
The product is
\[
AB = \begin{pmatrix} aa'+bc' & ab'+bd' \\ ca'+dc' & cb'+dd' \end{pmatrix}
=
\begin{pmatrix} u_{22}-u_{14} & u_{13}-u_{21} \\ u_{42}-u_{34} & u_{33}-u_{41} \end{pmatrix}
.
\]
The mixed trace in Fricke's identity is therefore
\[
\tr(AB)= u_{22}+u_{33}-u_{14}-u_{41}.
\]
We have already seen in Proposition~\ref{prop:charpoly} that
\[
(\tr{A})(\tr{B}) = \tr(u) = u_{11}+u_{22}+u_{33}+u_{44}.
\]
Therefore
\[
\big((\tr{A})(\tr{B}) - \tr(AB) \big) \tr(AB) = (u_{11}+u_{44}+u_{14}+u_{41})(u_{22}+u_{33}-u_{14}-u_{41})
.
\]
Before turning to the commutator, we can already begin comparing this to
\begin{align}
\sum_{i<j} (u_{ii}u_{jj}-u_{ij}u_{ji}) = &u_{11}u_{22}+u_{11}u_{33}+u_{22}u_{44}+u_{33}u_{44}-u_{14}u_{41}+ \nonumber \\
&u_{22}u_{33}-u_{12}u_{21}-u_{13}u_{31}-u_{24}u_{42}-u_{34}u_{43} + \label{ij}\\
&u_{11}u_{44}-u_{23}u_{32} . \nonumber
\end{align}
The first line contains terms that also appear after expanding
\[
 (u_{11}+u_{44}+u_{14}+u_{41})(u_{22}+u_{33}-u_{14}-u_{41}).
\]
The second line can be matched with terms that occur in $\tr(ABA^{-1}B^{-1})$, as we will see next.

Like $AB$, we have
\[
A^{-1}B^{-1} = \begin{pmatrix} dd'+bc' & -db'-ba' \\ -cd'-ac' & cb'+aa' \end{pmatrix} =
\begin{pmatrix} u_{33}-u_{14} & u_{43}-u_{24} \\ u_{12}-u_{31} & u_{22}-u_{41} \end{pmatrix}
.
\]
The commutator is
\[
ABA^{-1}B^{-1} = \begin{pmatrix} u_{22}-u_{14} & u_{13}-u_{21} \\ u_{42}-u_{34} & u_{33}-u_{41} \end{pmatrix}
\begin{pmatrix} u_{33}-u_{14} & u_{43}-u_{24} \\ u_{12}-u_{31} & u_{22}-u_{41} \end{pmatrix}
\]
with trace
\begin{align*}
\tr(ABA^{-1}B^{-1}) =&(u_{22}-u_{14})(u_{33}-u_{14})+(u_{13}-u_{21})(u_{12}-u_{31}) \\
&+ (u_{42}-u_{34})(u_{43}-u_{24})+(u_{33}-u_{41})(u_{22}-u_{41})
.
\end{align*}
Some of these terms cancel with $\big(\tr(A)\tr(B)-\tr(AB) \big)\tr(AB)$.
Others cancel with $\sum_{i<j} (u_{ii}u_{jj}-u_{ij}u_{ji})$ as stated before.
Let us indicate which by writing
\begin{align}
&\tr(ABA^{-1}B^{-1}) = \label{c-trace} \\
&u_{14}^2+u_{41}^2-u_{22}u_{14}-u_{33}u_{14}-u_{22}u_{41}-u_{33}u_{41}+ \nonumber \\
&2 u_{22}u_{33}-u_{12}u_{21}-u_{13}u_{31}-u_{24}u_{42}-u_{34}u_{43} + \nonumber \\
&u_{12}u_{13}+u_{21}u_{31}+u_{42}u_{43}+u_{34}u_{24} \nonumber
.
\end{align}
The first line occurs in $(u_{11}+u_{44}+u_{14}+u_{41})(u_{22}+u_{33}-u_{14}-u_{41})$ but with the opposite sign.
The second line of (\ref{c-trace}) matches the second line of (\ref{ij}).
After these cancellations, what we have to show is
\begin{align*}
u_{11}u_{44} - u_{23}u_{32} = &u_{22}u_{33} - u_{14}u_{41} \\
&-u_{44}u_{14}-u_{44}u_{41}-u_{11}u_{14}-u_{11}u_{41} \\
&+u_{24}u_{34}+u_{42}u_{43}+u_{12}u_{13} + u_{21}u_{31}
.
\end{align*}
This follows from the quadratic relations (\ref{eqn:9-quadratic-relations}) among the $u_{ij}$.
Indeed, we have written it this way because the top lines are already equal, and everything cancels out between the second and third lines.

On the top line,
\[
u_{11}u_{44} - u_{23}u_{32} = ad'da'-bb'cc' = aa'dd' - bc'cb' = u_{22}u_{33} - u_{14}u_{41}
.
\]
Between the second and third lines, we have for example
\[
u_{24}u_{34}-u_{44}u_{14}=-ba'dc'+da'bc' = 0.
\]
In the same way,
\begin{align*}
u_{42}u_{43}-u_{44}u_{41} &= -ca'db'+da'cb' = 0\\
u_{12}u_{13}-u_{11}u_{14} &= -ac'bd'+ad'bc' = 0 \\
u_{21}u_{31}-u_{11}u_{41} &= -ab'cd'+ad'cb' = 0
\end{align*}
completing the proof.
\qed

\section{Smaller examples} \label{sec:smaller-examples}

In this section, we explain how some easier identities can be seen as parallel to Fricke's.
These are the fact that, for $A$ in $\SL_2$,
\begin{equation} \label{eqn:trace-square}
\tr(A^2)  = (\tr{A})^2 - 2
\end{equation}
as well as the double-angle formula
\begin{equation}
\cos{2\theta} = 2(\cos{\theta})^2 - 1.
\end{equation}
The addition formula for $\cos(\theta+\theta')$ can also be seen as a case of Fricke's identity.

\subsection{The spin cover $\SL_2 \rightarrow SO(2,1)$}

The diagonal case $A=B$ of Fricke's identity is
\[
2(\tr{A})^2 + (\tr{A^2})^2 = (\tr{A})^2 \tr(A^2) + 4
\]
or
\[
(\tr{A^2}-2)(\tr{A^2}+2)=(\tr{A^2}-2)(\tr{A})^2.
\]
Assuming $\tr{A^2} \neq 2$ and working over an integral domain, cancelling gives (\ref{eqn:trace-square}).
Of course, it can also be verified directly without Fricke's identity:
\[
\tr{A^2} = a^2+d^2+2bc = (a+d)^2 - 2(ad-bc) = (\tr{A})^2-2.
\]
One can also obtain (\ref{eqn:trace-square}) by taking $X=Y=A$ in (\ref{eqn:trace-inverse}).

The identity (\ref{eqn:trace-square}) can also be interpreted as two ways of computing a characteristic polynomial, as in Section~\ref{sec:profricke}.
In terms of the characteristic polynomial, (\ref{eqn:trace-square}) says
\begin{equation} \label{eqn:trace-square1}
(\tr{A})^2-1 = \tr(A^2)+1.
\end{equation}
Let us work over a field of characteristic not equal to 2.
The diagonal $\SL_2 \subset \SL_2 \times \SL_2$ acts by conjugation $m \mapsto gmg^{-1}$, preserving both $\det(m)$ and $\tr(m)$.
It therefore preserves
\[
4\det(m) - \tr(m)^2
\]
which gives a map from $\SL_2$ to an orthogonal group in three variables.
Writing the quadratic form this way is convenient because
\begin{equation} \label{eqn:det-22}
ps-qr = \frac{1}{4}(4ps - 4qr) = \frac{1}{4}\big( (p+s)^2-(p-s)^2-(q+r)^2+(q-r)^2 \big) 
\end{equation}
where $p+s=\tr(m)$.
In this way, the diagonal $\SL_2$ preserves
\[
4\det(m)-\tr(m)^2 = -(p-s)^2-(q+r)^2+(q-r)^2.
\]

On the matrices of trace 0, the two forms $\det(m)$ and $\det(m)-\frac{1}{4}(\tr{m})^2$ coincide.
In the coordinates above, the subspace where $\tr{m}=0$ is given by $s=-p$.
The action
\begin{align*}
&\begin{pmatrix} a & b \\ c & d \end{pmatrix}\begin{pmatrix} p & q \\ r & -p \end{pmatrix}\begin{pmatrix} a & b \\ c & d \end{pmatrix}^{-1}
 \\
&=\begin{pmatrix} apd+brd-aqc+bpc & -apb-brb+aqa-bpa \\ cpd+drd-cqc+dpc & -cpb-drb+cqa-dpa \end{pmatrix}
\end{align*}
is represented by
\[
\begin{pmatrix} p \\ q \\ r \end{pmatrix} \mapsto \begin{pmatrix} ad+bc & -ac & bd \\ -2ab & a^2 & -b^2 \\ 2cd & -c^2 & d^2 \end{pmatrix} \begin{pmatrix} p \\ q \\ r \end{pmatrix}
.
\]

\begin{proposition}
For $A$ in $\SL_2$ acting by conjugation on the matrices of trace $0$, the characteristic polynomial is
\[
x^3 - \big( (\tr{A})^2-1 \big)x^2 +\big( (\tr{A})^2-1 \big) x - 1
.
\]
\end{proposition}
\begin{proof}
Direct calculation, without even using the relation $ad-bc=1$, gives
\[
x^3 - \big( (\tr{A})^2-\det{A}\big)x^2 + \big( \det(A) (\tr{A})^2 - (\det{A})^2 \big)x - (\det{A})^3.
\]
\end{proof}

\subsection{Double-angle formula}

The trigonometric identity
\begin{equation} \label{eqn:so2}
2\cos^2{\theta} - 1 = \cos{2\theta}
\end{equation}
can be seen as an instance of (\ref{eqn:trace-square})
where $A = \begin{pmatrix} \cos{\theta} & -\sin{\theta} \\ \sin{\theta} & \cos{\theta} \end{pmatrix}$ is a rotation matrix. 
However, $A^2$ is the rotation by twice the angle. There is a factor of 2 between (\ref{eqn:so2}) and what directly comes out of (\ref{eqn:trace-square}), as in $(2\cos{\theta})^2-2 =2\cos{2\theta}$.
This is reminiscent of the extra factors that arise when taking $A=B$ in Fricke's identity to deduce (\ref{eqn:trace-square}).

The connection goes somewhat further than simply applying (\ref{eqn:trace-square}) to a particular matrix.
Again, one can state it in terms of characteristic polynomials.
We restrict $\det(m)$ further to the subspace where, as before, $\tr(m)=0$ and now also $q=r$.
This is a 2-dimensional space of matrices
\[
m = p \begin{pmatrix} 1 & 0 \\ 0 & -1 \end{pmatrix} + q \begin{pmatrix} 0 & 1 \\ 1 & 0 \end{pmatrix}
\]
with a quadratic form
\[
\det(m) = -p^2 - q^2.
\]
The condition $q=r$ is not preserved under conjugation by $\SL_2$, but only by a 1-dimensional group.
We have
\begin{align*}
&\begin{pmatrix} a & b \\ c & d \end{pmatrix} \begin{pmatrix} p & q \\ q & -p \end{pmatrix} \begin{pmatrix} d & -b \\ -c & a \end{pmatrix} \\
&=\begin{pmatrix} (ad+bc)p+(bd-ac)q & -2abp + (a^2-b^2)q \\ 2cdp + (d^2-c^2)q & -(ad+bc)p-(bd-ac)q \end{pmatrix}
.
\end{align*}
The off-diagonal entries are equal to each other, assuming 3 equations in the 4 variables $a,b,c,d$ are satisfied:
\[
\left\{
\begin{aligned}
cd &=-ab \\
d^2-c^2 &= a^2-b^2 \\
ad-bc &=1
\end{aligned}
\right.
\]
These define a subgroup
\[
\begin{pmatrix} a & b \\ c & d \end{pmatrix} = \begin{pmatrix} a & \mp c \\ c & \pm a \end{pmatrix},
\quad a^2+c^2 = 1
\]
and we may as well consider the identity component $\begin{pmatrix} a & -c \\ c & a \end{pmatrix}$.
This subgroup acts on $\begin{pmatrix} p \\ q \end{pmatrix}$ by
\[
\begin{pmatrix} ad+bc & bd-ac \\ -2ab & a^2-b^2 \end{pmatrix} = \begin{pmatrix} a^2-c^2 & -2ac \\ 2ac & a^2-c^2 \end{pmatrix}.
\]

Since $a^2+c^2=1$, this matrix has characteristic polynomial
\[
x^2 - 2(2a^2-1) x + 1.
\]
Setting $a=\cos{\theta}$ and $c=\sin{\theta}$, this matrix becomes
\[
\begin{pmatrix} \cos{2\theta} & -\sin{2\theta} \\ \sin{2\theta} & \cos{2\theta} \end{pmatrix}
= A^2
\]
where $A=\begin{pmatrix} a & b \\ c & d \end{pmatrix}$ is the rotation matrix of half the angle, with trace $2\cos{\theta}$.

Like $(\tr{A})^2+(\tr{B})^2-2$ in Fricke's identity, one side of (\ref{eqn:so2}) appears as the middle coefficient of the characteristic polynomial of the image in an orthogonal group.

\subsection{Addition formula for cosines}

Taking both $A$ and $B$ to be rotation matrices, we find from Fricke's identity that
\[
(2\cos{\theta})^2+(2\cos{\theta'})^2 + \big(2\cos(\theta + \theta')\big)^2 = 8 \cos{\theta} \cos{\theta'} \cos(\theta+\theta') + 4
.
\]
This is a quadratic equation for $\cos(\theta+\theta')$, with solution
\[
\cos(\theta + \theta') = \cos{\theta} \cos{\theta'} \pm \sqrt{(1-\cos^2{\theta})(1-\cos^2{\theta'})}
\]
which agrees with the usual formula in terms of $\sin{\theta}=\sqrt{1-\cos^2{\theta}}$.

\section{Larger examples} \label{sec:symplectic}

\subsection{An identity for symplectic matrices}
The symplectic form given by the matrix $\begin{pmatrix} & I \\ -I & \end{pmatrix}$ is preserved by $2n \times 2n$ matrices of the form
$
\begin{pmatrix} A & B \\ C & D \end{pmatrix}
$
where the blocks $A,B,C,D$ are $n \times n$ and satisfy
\begin{equation} \label{eqn:symplectic-abcd}
\left\{ 
\begin{aligned}
A^{\top}C &= C^{\top} A \\
B^{\top}D &= D^{\top} B \\
A^{\top}D - C^{\top}B &= I
\end{aligned}
\right.
\end{equation}
Indeed, a block matrix takes the symplectic form to
\[
\begin{pmatrix} A & B \\ C & D \end{pmatrix}^{\top}
\begin{pmatrix} & I \\ -I & \end{pmatrix}
\begin{pmatrix} A & B \\ C & D\end{pmatrix}
=
\begin{pmatrix}
A^{\top}C-C^{\top}A & A^{\top}D-C^{\top}B \\ -(A^{\top}D-C^{\top}B)^{\top} & B^{\top}D - D^{\top}B
\end{pmatrix}
.
\]
Here and below, we use the convention that an empty entry in a matrix denotes a block matrix $0$ of the appropriate size.

The next proposition generalises (\ref{eqn:trace-square}).
For a related identity, see \cite[Corollary 4.2]{GKW} where the authors obtain a generalisation of (\ref{eqn:trace-inverse}) for matrices over non-commutative rings.
\begin{proposition} \label{prop:symplectic}
For any $2n \times 2n$ symplectic matrix $\begin{pmatrix} A & B \\ C & D \end{pmatrix}$
\begin{equation} \label{eqn:symplectic}
\tr \Big( \begin{pmatrix} A^{\top} & B^{\top} \\ C^{\top} & D^{\top} \end{pmatrix}\begin{pmatrix} A & B \\ C & D \end{pmatrix} \Big) = \tr\big( (A+D)^{\top} (A+D) \big) - 2n
.
\end{equation}
\end{proposition}
\begin{proof}
In block form, one can calculate as in the $2 \times 2$ case that
\[
\begin{pmatrix} A^{\top} & B^{\top} \\ C^{\top} & D^{\top} \end{pmatrix}\begin{pmatrix} A & B \\ C & D \end{pmatrix}
=
\begin{pmatrix}
A^\top A+B^\top C & A^\top B + B^\top D \\ C^\top A+ D^\top C& C^\top B + D^\top D
\end{pmatrix}
.
\]
Substituting $A^{\top}D - C^{\top}B = I$ and its transpose, we obtain
\[
\begin{pmatrix} A^{\top} & B^{\top} \\ C^{\top} & D^{\top} \end{pmatrix}\begin{pmatrix} A & B \\ C & D \end{pmatrix}
=
\begin{pmatrix}
A^\top A+D^\top A-I & A^\top B + B^\top D \\ C^\top A+ D^\top C& A^\top D-I + D^\top D
\end{pmatrix}
.
\]
The trace of this matrix is
\begin{align*}
&\tr(A^\top A + D^\top A-I) + \tr(A^\top D + D^\top D - I) = \\
&\tr(A^\top (A+D) + D^\top (A+D)) - 2\tr(I) = \tr\big( (A+D)^\top (A+D) \big) - 2n.
\end{align*}
\end{proof}

For $n=1$, we have $A^\top=A$ and so on for all of these $1 \times 1$ matrices.
The identity (\ref{eqn:symplectic}) reduces to
$
\tr(g^2) = (\tr{g})^2-2
$
for $g \in \SL_2$, that is, (\ref{eqn:trace-square}).

\begin{proposition} \label{prop:symplectic-double}
If $\begin{pmatrix} a_1 & b_1 \\ c_1 & d_1 \end{pmatrix}$ and $\begin{pmatrix} a_2 & b_2 \\ c_2 & d_2 \end{pmatrix}$ are symplectic $2n \times 2n$ matrices, then
\[
\begin{pmatrix} a_1 & & b_1 & \\
& a_2 & & b_2 \\
c_1 & & d_1 \\
& c_2 & & d_2
\end{pmatrix}
\]
is a symplectic $4n \times 4n$ matrix.
\end{proposition}
\begin{proof}
Since $a_1^{\top}c_1 = c_1^{\top} a_1$, we have
\[
\begin{pmatrix} a_1^{\top} & \\  &a _2^{\top}  \end{pmatrix} \begin{pmatrix}  c_1 & \\  &  c_2 \end{pmatrix}
=
\begin{pmatrix} c_1^{\top} & \\  &c _2^{\top}  \end{pmatrix} \begin{pmatrix}  a_1 & \\  &  a_2 \end{pmatrix}
=\begin{pmatrix} c_1 & \\ & c_2 \end{pmatrix}^{\top} \begin{pmatrix}  a_1 & \\  &  a_2 \end{pmatrix}
.
\]
Similarly for $b$ and $d$. Finally,
\begin{align*}
A^{\top}D-C^{\top}B &= \begin{pmatrix} a_1^{\top} & \\ & a_2^{\top} \end{pmatrix}
\begin{pmatrix} d_1 & \\ & d_2 \end{pmatrix}
-
\begin{pmatrix} c_1^{\top} & \\ & c_2^{\top} \end{pmatrix}
\begin{pmatrix} b_1 & \\ & b_2 \end{pmatrix} \\
&= \begin{pmatrix} a_1^{\top}d_1 - c_1^{\top}b_1 & \\ & a_2^{\top}d_2-c_2^{\top}b_2 \end{pmatrix}
\\
&= \begin{pmatrix} I & \\ & I \end{pmatrix}
.
\end{align*}
\end{proof}
The identity (\ref{eqn:symplectic}) can be applied to $\begin{pmatrix} a & b \\ c & d \end{pmatrix}$ or to $\begin{pmatrix} a & & b & \\ & a & & b \\ c & & d & \\ & c & & d \end{pmatrix}$, and the results are the same after cancelling a factor of 2, similar to the extra factors from Section~\ref{sec:smaller-examples}.

Apply (\ref{eqn:symplectic}) to $\begin{pmatrix} a & & b & \\ & a' & & b'\\ c & & d & \\ & c' & & d' \end{pmatrix}$, where $n=2$ and $g=\begin{pmatrix} a & b \\ c & d \end{pmatrix}$ and $g'=\begin{pmatrix} a' & b' \\  c' & d' \end{pmatrix}$ both belong to $\SL_2$.
The quantity on the right becomes
\[
\tr\big( (A+D)^{\top}(A+D) \big) -2n = \tr \begin{pmatrix} (a+d)^2 & \\ & (a'+d')^2 \end{pmatrix}-4 = (\tr{g})^2+(\tr{g'})^2-4
\]
We know from (\ref{eqn:fricke}) that this must equal
\[
-(\tr{gg'})^2+(\tr{g})(\tr{g})(\tr{gg'})+\tr{gg'g^{-1}g'^{-1}}-2.
\]
Reversing the logic, one could start from here and follow this as another route to proving Fricke's identity.

\subsection{The final exceptional spin group}

Let us work over a field $k$.
The special linear group $\SL_{4n}(k)$ preserves the $4n\times 4n$ determinant, hence preserves the bilinear form on $\bigwedge^{2n} k^{4n}$ given by
\[
\langle v_1 \wedge \ldots \wedge v_{2n}, v_{2n+1} \wedge \ldots \wedge v_{4n} \rangle = \det \begin{pmatrix} v_1 & \cdots & v_{4n} \end{pmatrix}
\]
and extended by linearity.
Here we think of $v_1,\ldots,v_{4n}$ as column vectors, and the form is symmetric because $2n$ is even.
We compose the map from Proposition~\ref{prop:symplectic-double} with the inclusion of $\Sp_{4n}$ into $\SL_{4n}$ in order to obtain an orthogonal representation:
\[
\Sp_{2n} \times \Sp_{2n} \rightarrow \Sp_{4n} \rightarrow \SL_{4n} \rightarrow O\big({\bigwedge}^{2n} k^{4n} \big)
\]

For $n=1$, we have
\[
\dim \SL_{4} = \dim O\big( {\bigwedge}^2 k^4 \big) = 15
\]
but for $n > 1$ the orthogonal group is much larger than $\SL_{4n}$.
\begin{center}
\begin{tabular}{c|c|c}
$n$ & $\dim \SL_{4n} = (4n)^2-1$ & $\dim O(\wedge^{2n} k^{4n})= \binom{\binom{4n}{2n}}{2}$ \\ \hline
1 & 15 & 15 \\
2 & 63 & 2415 \\
3 & 143 & 426426 \\
\end{tabular}
\end{center}
%gp > for(n=1,5,print([n,16*n^2-1,binomial(binomial(4*n,2*n),2)]))
%[1, 15, 15]
%[2, 63, 2415]
%[3, 143, 426426]
%[4, 255, 82812015]
%[5, 399, 17067297390]

Let $e_1,\ldots,e_4$ be a basis for $k^4$.
The pure wedges $e_i \wedge e_j$ form a basis for $\bigwedge^2 k^4$, which we put in the lexicographic order $12<13<14<23<24<34$.
In this basis, $g$ acts by a matrix with entries
\[
(\wedge^2 g)_{\alpha \beta,ij} = g_{\alpha i}g_{\beta j}-g_{\beta i}g_{\alpha j}
.
\]
In matrix form, for matrices of the form
$
g=\begin{pmatrix} a & & b & \\ & a' & & b' \\ c & & d & \\ & c' & & d' \end{pmatrix}
\in \SL_2 \times \SL_2 \subset \Sp_4
$
we find that
\[
\wedge^2 g = \begin{pmatrix} aa' & & ab' & -ba' & & bb' \\
& 1 & & & & & \\
ac' & & ad' & -bc' & & bd' \\
-ca' & & -cb' & da' & & -db' \\
&&&& 1 & \\
cc' & & cd' & -dc' & & dd'
\end{pmatrix}
\]
which differs somewhat from (\ref{eqn:u}).
Still, the entries obey quadratic relations like (\ref{eqn:9-quadratic-relations}) and one could try $\wedge^2 g$ instead of $u$ to give a similar proof of Fricke's identity.

\section*{Acknowledgements}

Thanks to the Banff International Research Station, where this work was first presented at the workshop
\href{https://www.birs.ca/events/2025/5-day-workshops/25w5411}{\textit{Perspectives on Markov Numbers}} (\url{https://www.birs.ca/events/2025/5-day-workshops/25w5411}).
Thanks to Dani Kaufman for a helpful conversation at BIRS pointing out reference \cite{GKW}.

MdCI was supported by the Knut and Alice Wallenberg Foundation (Wallenberg Initiative for Networks and Quantum Information) and by Stiftelsen Hierta Retzius through the Royal Swedish Academy of Sciences.

\end{document}